\documentclass[12pt]{article}
\usepackage{amsmath,amssymb,amsthm}
\usepackage{hyperref}
\usepackage{datetime}
\usepackage{units}
\usepackage{color}
\usepackage[T1]{fontenc}
\usepackage[utf8]{inputenc}
\usepackage{authblk}
\usepackage{bm,latexsym,mathrsfs,enumerate}
\usepackage{color}

\title{The golden ratio, Fibonacci numbers and BBP-type formulas\thanks{%
MSC 2010: 11B39, 11Y60}}
\author[]{Kunle Adegoke\thanks{adegoke00@gmail.com }}
\affil{Department of Physics and Engineering Physics, \mbox{Obafemi Awolowo University, Ile-Ife, 220005 Nigeria}}

\theoremstyle{plain}
\numberwithin{equation}{section}
\newtheorem{thm}{THEOREM}[section]

\newtheorem{rem}[thm]{Remark}

\begin{document}
\date{}
\maketitle
%\today, \currenttime
\begin{abstract}
\noindent We derive interesting arctangent identities involving the golden ratio, \mbox{Fibonacci} numbers and \mbox{Lucas} numbers. Binary BBP-type formulas for the arctangents of certain odd powers of the golden ratio are also derived, for the first time in the literature. Finally we derive golden-ratio-base BBP-type formulas for some mathematical constants, including $\pi$, $\log 2$, $\log\phi$ and $\sqrt 2\,\arctan\sqrt 2$. The \mbox{$\phi-$nary} BBP-type formulas derived here are considerably simpler than similar results contained in earlier literature. 
\end{abstract}
\tableofcontents

\section{Introduction}
This paper is concerned with the derivation of interesting arctangent identities connecting the golden ratio, Fibonacci numbers and the related Lucas numbers. Binary BBP-type formulas for arctangents of the odd powers of the golden ratio will be derived, as well as golden-ratio-base BBP-type formulas for some mathematical constants. We will also present a couple of base~$5$ BBP-type formulas for linear combinations of the arctangents of even powers of the golden ratio. The golden ratio, having the numerical value of $(\sqrt 5+1)/2$ is denoted throughout this paper by $\phi$. The Fibonacci numbers are defined, as usual, through the recurrence relation $F_n=F_{n-1}+F_{n-2}$, with $F_0=0$ and $F_1=1$. The Lucas numbers are defined by $L_n=F_{n-1}+F_{n+1}$.

We shall often make use of the following algebraic properties of $\phi$

\begin{subequations}
\begin{eqnarray}\label{equ.xj19lbd}
\phi^2 &=& 1+\phi\,,\label{equ.zuch1o9}\\
\sqrt 5 &=& 2\phi-1\,,\label{equ.ln4di53}\\
\phi-1 &=& 1/\phi\,,\label{equ.epgagl8}\\
\phi^n &=& \phi F_n+F_{n-1}\,,\label{equ.n1xmwpl}\\
\phi^{-n} &=& (-1)^n(-\phi F_n+F_{n+1})\,,\label{equ.wr334af}\\
\mbox{and}\nonumber\\
\phi^n &=& \phi^{n-1}+\phi^{n-2}\label{equ.gxbvppb}\,.
\end{eqnarray}
\end{subequations}

% \thispagestyle{fancy}

% \vfil\eject
% \fancyhead{}
% \fancyhead[CO]{\hfill THE GOLDEN RATIO, FIBONACCI NUMBERS AND BBP-TYPE FORMULAS}
% \fancyhead[CE]{THE FIBONACCI QUARTERLY \hfill}
% \renewcommand{\headrulewidth}{0pt}

We will also need the following trigonometric identities

\begin{subequations}
\begin{eqnarray}
\tan^{-1} x+\tan^{-1} y &=& \tan^{-1}\left (\frac{x+y}{1-xy}\right )\,,\quad xy<1\label{equ.ohqvcqn}\\
\tan^{-1} x-\tan^{-1} y &=& \tan^{-1}\left (\frac{x-y}{1+xy}\right )\,,\quad xy>-1\,\label{equ.gdoxus5}.
\end{eqnarray}
\end{subequations}

Results for arctangent identities involving the Fibonacci numbers and related sequences can also be found in earlier references~\cite{hoggatt5, shannon, mahon} and references therein, while results for \mbox{$\phi-$nary} BBP type formulas can be found in references~\cite{bailey01, borwein, chan, zhang}.

\section{Arctangent formulas for the odd powers of the golden ratio}

In this section we will present results for the arctangents of the odd powers of the golden ratio in terms of the arctangents of reciprocal Fibonacci and reciprocal Lucas numbers, as well as in terms of the arctangents of consecutive Fibonacci numbers.

\subsection{Arctangent formulas involving reciprocal Fibonacci and reciprocal Lucas numbers}

\begin{thm}\label{thmma.ahffbvn}

For positive integers, $k$,
\begin{equation}\label{equ.svdrzxs}
\tan^{-1}\phi^{2k-1}=2\tan^{-1}1-\frac{1}{2}\tan^{-1}\left(\frac{2}{L_{2k-1}}\right)\,.
\end{equation}

\end{thm}

\begin{proof}
Choosing $x=1/\phi^{2k-1}=y$ in~\eqref{equ.ohqvcqn}, we find
\[
\begin{split}
& 2\tan^{-1}\frac{1}{\phi^{2k-1}}\\
& =\tan^{-1}\left(\frac{2}{\phi^{2k-1}-\phi^{-(2k-1)}}\right)\\
& =\tan^{-1}\left(\frac{2}{F_{2k-2}+F_{2k}}\right)\\
& =\tan^{-1}\left(\frac{2}{L_{2k-1}}\right)\,,
\end{split}
\]

and the result follows.

\end{proof}

\begin{thm}

For non-zero integers, $k$,
\begin{subequations}
\begin{eqnarray}
&& \tan^{-1}\phi^{2k+1}=2\tan^{-1}1+\frac{1}{2}\tan^{-1}\left(\frac{1}{L_{2k}}\right)-\frac{1}{2}\tan^{-1}\left(\frac{1}{F_{2k}}\right)\label{equ.fsdxekk}\\
\nonumber\\
&& \tan^{-1}\phi^{2k-1}=2\tan^{-1}1-\frac{1}{2}\tan^{-1}\left(\frac{1}{L_{2k}}\right)-\frac{1}{2}\tan^{-1}\left(\frac{1}{F_{2k}}\right)\label{equ.w7urgvy}
\end{eqnarray}
\end{subequations}

\end{thm}

\begin{proof}
Choosing $x=1/\phi^{2k-1}$ and $y=1/\phi^{2k+1}$ in~\eqref{equ.ohqvcqn} and using the algebraic properties of $\phi$, it is straightforward to establish that
\[
\begin{split}
&\tan ^{ - 1} \left( {\frac{1}{{\phi ^{2k - 1} }}} \right) + \tan ^{ - 1} \left( {\frac{1}{{\phi ^{2k + 1} }}} \right)\\
& = \tan ^{ - 1} \left( {\frac{{\phi  + \phi^{-1}}}{{\phi ^{2k}  - \phi ^{ - 2k} }}} \right)\\
& = \tan ^{ - 1} \frac{{\sqrt 5 }}{{\sqrt 5 F_{2k} }} = \tan ^{ - 1} \left( {\frac{1}{{F_{2k} }}} \right)\,.
\end{split}
\]

That is
\begin{equation}\label{equ.ehmv8kg}
\tan^{-1}\phi^{2k+1}+\tan^{-1}\phi^{2k-1}=\pi-\tan ^{ - 1} \left( {\frac{1}{{F_{2k} }}} \right)\,.
\end{equation}

Choosing $x=\phi^{2k+1}$ and $y=\phi^{2k-1}$ in~\eqref{equ.gdoxus5}, we find

\begin{equation}\label{equ.fgingzp}
\begin{split}
\tan ^{ - 1} \phi ^{2k + 1}  - \tan ^{ - 1} \phi ^{2k - 1}&= \tan ^{ - 1} \left( {\frac{{\phi  - \phi ^{ - 1} }}{{\phi ^{ - 2k}  + \phi ^{2k} }}} \right)\\
&= \tan ^{ - 1} \left( {\frac{1}{{L_{2k} }}} \right)
\end{split}
\end{equation}

Addition of Eqs.~\eqref{equ.ehmv8kg} and \eqref{equ.fgingzp} gives~\eqref{equ.fsdxekk}, while subtraction of~\eqref{equ.fgingzp} from~\eqref{equ.ehmv8kg} gives~\eqref{equ.w7urgvy}. 
\end{proof}

\begin{rem}
Note that by performing the telescoping summation suggested by~\eqref{equ.fgingzp}, it is established that
\[
\tan^{-1} \phi^{2n+1} =\tan^{-1}\phi+\sum_{k = 1}^{n}\tan^{-1}\left(\frac{1}{L_{2k}}\right)\,,
\]
\end{rem}

which can be written as
\begin{equation}
\tan^{-1} \phi^{2n+1} = \tan^{-1}1+\frac{1}{2}\tan^{-1}\frac{1}{2}+\sum_{k = 1}^{n}\tan^{-1}\left(\frac{1}{L_{2k}}\right)\,,\label{equ.tjjg38v}
\end{equation}
since

\begin{equation}\label{equ.nhfkxe6}
\tan^{-1}\phi=\tan^{-1}1+\frac{1}{2}\tan^{-1}\frac{1}{2}\quad (k=1 \text{ in Theorem~\ref{thmma.ahffbvn}})\,.
\end{equation}

\subsection{Arctangent formulas involving the ratio of consecutive \mbox{Fibonacci} numbers}

\begin{thm}

For non-negative integers, $n$,
\begin{subequations}
\begin{eqnarray}
&& \tan^{-1} \phi^{4n-1} = 3\tan^{-1}1-\frac{1}{2}\tan^{-1}\frac{1}{2}-\tan^{-1}\left(\frac{F_{2n-1}}{F_{2n}}\right)\label{equ.ric8gsn},\\
\nonumber\\
&& \tan^{-1} \phi^{4n-3} = \tan^{-1}1+\frac{1}{2}\tan^{-1}\frac{1}{2}+\tan^{-1}\left(\frac{F_{2n-2}}{F_{2n-1}}\right)\,.\label{equ.g6b2dzg}
\end{eqnarray}
\end{subequations}

\begin{proof}
Choosing $x=1/\phi$ and $y=F_{p-1}/F_p$ in the trigonometric identity~\eqref{equ.gdoxus5}, clearing fractions and using properties~\eqref{equ.n1xmwpl} and \eqref{equ.wr334af} to simplify the numerator and denominator of the arctangent argument, and replacing $\tan^{-1}\phi$ with the right hand side of~\eqref{equ.nhfkxe6}, we obtain

\begin{equation}\label{equ.q77es1t}
(-1)^p\tan^{-1}(\phi^{2p-1})=\left(2(-1)^p+1\right)\tan^{-1}1-\frac{1}{2}\tan^{-1}\frac{1}{2}-\tan^{-1}\frac{F_{p-1}}{F_p}\,.
\end{equation}

Setting $p=2n$ in~\eqref{equ.q77es1t}, we obtain identity \eqref{equ.ric8gsn}, while $p=2n-1$ in~\eqref{equ.q77es1t} gives identity \eqref{equ.g6b2dzg}.
\end{proof}

\end{thm}

\section{Arctangent formulas for the reciprocal even powers of the golden ratio}

\begin{thm}
For non-negative integers, $k$,

\begin{equation}\label{equ.qmp0021}
2\tan^{-1}\left(\frac{1}{\phi^{2k}}\right)=\tan^{-1}\left(\frac{2}{F_{2k}\sqrt 5}\right)\,.
\end{equation}

\end{thm}

\begin{proof}
Identity~\eqref{equ.qmp0021} follows from the choice of $x=1/\phi^{2k}=y$ in~\eqref{equ.ohqvcqn}.

\end{proof}

\begin{thm}
For non-negative integers, $k$,

\begin{subequations}
\begin{eqnarray}
&& 2\tan^{-1}\left(\frac{1}{\phi^{2k}}\right)=\tan^{-1}\left(\frac{\sqrt 5}{L_{2k+1}}\right)+\tan^{-1}\left(\frac{1}{F_{2k+1}\sqrt 5}\right)\label{equ.ydokq3d}\\
\nonumber\\
&& 2\tan^{-1}\left(\frac{1}{\phi^{2k+2}}\right)=\tan^{-1}\left(\frac{\sqrt 5}{L_{2k+1}}\right)-\tan^{-1}\left(\frac{1}{F_{2k+1}\sqrt 5}\right)\,.\label{equ.eggvqz7}
\end{eqnarray}
\end{subequations}

\end{thm}

\begin{proof}
Choosing $x=1/\phi^{2k}$ and $y=1/\phi^{2k+2}$ in~\eqref{equ.ohqvcqn} and using the algebraic properties of $\phi$, it is straightforward to establish that

\begin{equation}\label{equ.uzbbhnp}
\tan^{-1}\left(\frac{1}{\phi^{2k}}\right)+\tan^{-1}\left(\frac{1}{\phi^{2k+2}}\right)=\tan^{-1}\left(\frac{\sqrt 5}{L_{2k+1}}\right)\,.
\end{equation}

Choosing $x=1/\phi^{2k}$ and $y=1/\phi^{2k+2}$ in~\eqref{equ.gdoxus5}, we find

\begin{equation}\label{equ.h37ks3p}
\tan^{-1}\left(\frac{1}{\phi^{2k}}\right)-\tan^{-1}\left(\frac{1}{\phi^{2k+2}}\right)=\tan^{-1}\left(\frac{1}{F_{2k+1}\sqrt 5}\right)\,.
\end{equation}

Addition of~\eqref{equ.uzbbhnp} and \eqref{equ.h37ks3p} gives~\eqref{equ.ydokq3d}, while subtraction of~\eqref{equ.h37ks3p} from~\eqref{equ.uzbbhnp} gives~\eqref{equ.eggvqz7}.

\end{proof}

\begin{rem}
Performing the telescoping summation invited by~\eqref{equ.h37ks3p}, we obtain

\begin{equation}
\tan^{-1}\left(\frac{1}{\phi^{2}}\right)-\tan^{-1}\left(\frac{1}{\phi^{2n+2}}\right)=\sum_{k=1}^n\tan^{-1}\left(\frac{1}{F_{2k+1}\sqrt 5}\right)\,.
\end{equation}

Taking limit $n\to\infty$, we obtain the formula

\begin{equation}
\tan^{-1}\left(\frac{1}{\phi^{2}}\right)=\sum_{k=1}^\infty\tan^{-1}\left(\frac{1}{F_{2k+1}\sqrt 5}\right)\,.
\end{equation}

\end{rem}

\section{Arctangent identities involving the Fibonacci and Lucas numbers}
\subsection{Arctangent formulas involving reciprocal \mbox{Fibonacci} and \mbox{reciprocal} \mbox{Lucas} numbers}
\begin{thm}
For non-zero integers, $n$,
\begin{subequations}\label{equ.ptld4p6}
\begin{eqnarray}
\tan^{-1}\left(\frac{1}{F_{2n-2}}\right) &=& \tan^{-1}\left(\frac{1}{F_{2n}}\right)+\tan^{-1}\left(\frac{1}{L_{2n-2}}\right)+\tan^{-1}\left(\frac{1}{L_{2n}}\right)\,,\label{equ.x2ffu2e}\\
\nonumber\\
\nonumber\\
\tan^{-1}\left(\frac{2}{L_{2n-1}}\right) &=& \tan^{-1}\left(\frac{1}{F_{2n}}\right)+\tan^{-1}\left(\frac{1}{L_{2n}}\right)\,,\label{equ.myyri84}\\
\nonumber\\
\nonumber\\
\tan^{-1}\left(\frac{2}{L_{2n+1}}\right) &=& \tan^{-1}\left(\frac{1}{F_{2n}}\right)-\tan^{-1}\left(\frac{1}{L_{2n}}\right)\,,\label{equ.kfqaolk}\\
\nonumber\\
\tan^{-1}\left(\frac{2}{F_{2n}\sqrt 5}\right) &=& \tan^{-1}\left(\frac{\sqrt 5}{L_{2n+1}}\right)+\tan^{-1}\left(\frac{1}{F_{2n+1}\sqrt 5}\right)\,.\label{equ.dnn6d12}
\end{eqnarray}
\end{subequations}

\end{thm}

\pagebreak

% \begin{proof}
% Equation~\eqref{equ.x2ffu2e} is proved by setting $k=n-1$ in equation~\eqref{equ.fsdxekk} and $k=n$ in equation~\eqref{equ.w7urgvy} and eliminating $\tan^{-1}(\phi^{2n-1})$ between the two equations that result. Equation~\eqref{equ.myyri84} is proved by setting $k=n$ in equation~\eqref{equ.svdrzxs} and $k=n$ in equation~\eqref{equ.w7urgvy} and eliminating $\tan^{-1}(\phi^{2n-1})$ between the two equations that result. Equation~\eqref{equ.kfqaolk} is proved by setting $k=n+1$ in equation~\eqref{equ.svdrzxs} and $k=n$ in equation~\eqref{equ.fsdxekk} and eliminating $\tan^{-1}(\phi^{2n+1})$ between the two equations that result. Lastly, equation~\eqref{equ.dnn6d12} is proved by setting $k=n+1$ in equation~\eqref{equ.eggvqz7} and $k=n$ in equation~\eqref{equ.ydokq3d} and eliminating $\tan^{-1}(1/\phi^{2n})$ between the resulting equations.

% \end{proof}

\begin{rem}
Using the following identity (Theorem 4 of \cite{hoggatt5}), 

\begin{equation}\label{equ.wdgo3gg}
\tan^{-1}\left(\frac{1}{F_{2n}}\right) = \tan^{-1}\left(\frac{1}{F_{2n+1}}\right)+\tan^{-1}\left(\frac{1}{F_{2n+2}}\right)\,,
\end{equation}

identities~\eqref{equ.x2ffu2e}~---~\eqref{equ.kfqaolk} can also be written
\begin{subequations}
\begin{eqnarray}
\tan^{-1}\left(\frac{1}{F_{2n-1}}\right) &=& \tan^{-1}\left(\frac{1}{L_{2n-2}}\right)+\tan^{-1}\left(\frac{1}{L_{2n}}\right)\,,\label{equ.kzwxgqw}\\
\nonumber\\
\tan^{-1}\left(\frac{2}{L_{2n-1}}\right) &=& \tan^{-1}\left(\frac{1}{L_{2n}}\right)+\tan^{-1}\left(\frac{1}{F_{2n+1}}\right)+\tan^{-1}\left(\frac{1}{F_{2n+2}}\right)\,,\nonumber\\
\nonumber\\
\tan^{-1}\left(\frac{2}{L_{2n+1}}\right) &=& \tan^{-1}\left(\frac{1}{F_{2n+1}}\right)+\tan^{-1}\left(\frac{1}{F_{2n+2}}\right)-\tan^{-1}\left(\frac{1}{L_{2n}}\right)\nonumber\,.
\end{eqnarray}
\end{subequations}

The identity~\eqref{equ.kzwxgqw} is Theorem~3 of~\cite{hoggatt5}.

\end{rem}

%\clearpage

Subtracting equation~\eqref{equ.kfqaolk} from equation~\eqref{equ.myyri84} we obtain, for non-zero integers, the following arctangent identity involving three consecutive Lucas numbers:

\begin{thm}

\begin{equation}
\tan^{-1}\left(\frac{2}{L_{2n-1}}\right) = 2\tan^{-1}\left(\frac{1}{L_{2n}}\right)+\tan^{-1}\left(\frac{2}{L_{2n+1}}\right)\label{equ.szjec8z}\,.
\end{equation}

\end{thm}

\begin{rem}

It is instructive to compare the two identities equation~\eqref{equ.wdgo3gg} involving \mbox{Fibonacci} numbers and equation~\eqref{equ.szjec8z} involving Lucas numbers.

\end{rem}

\subsection{Arctangent formulas involving the ratio of consecutive \mbox{Fibonacci} numbers}

\begin{thm}
For positive integers, $n$,
\begin{subequations}\label{equ.nc2rp6z}
\begin{eqnarray}
\tan^{-1}\left(\frac{F_{2n}}{F_{2n+1}}\right)&=&\sum_{k=1}^{2n}\tan^{-1}\left( \frac{1}{L_{2k}}\right)\label{equ.k1d613q}\,,\\
\nonumber\\
\tan^{-1}\left (\frac{F_{2n-1}}{F_{2n}}\right)&=&\tan^{-1}2-\sum_{k=1}^{2n-1}\tan^{-1}\left( \frac{1}{L_{2k}}\right)\,,\label{equ.tbk36la}\\
\tan^{-1}\left(\frac{F_{2n}}{F_{2n+1}}\right)&=&\tan^{-1}1-\frac{1}{2}\tan^{-1}\frac{1}{2}-\frac{1}{2}\tan^{-1}\left(\frac{2}{L_{4n+1}}\right)\label{equ.r1jgvwv}\\
\mbox{and}\nonumber\\
\tan^{-1}\left(\frac{F_{2n-1}}{F_{2n}}\right)&=&\tan^{-1}1-\frac{1}{2}\tan^{-1}\frac{1}{2}+\frac{1}{2}\tan^{-1}\left(\frac{2}{L_{4n-1}}\right)\,.\label{equ.e1vvx0a}
\end{eqnarray}
\end{subequations}

\end{thm}

\begin{proof}
Comparing identities~\eqref{equ.g6b2dzg} and \eqref{equ.tjjg38v}, we obtain identity~\eqref{equ.k1d613q}, expressing, as a sum of reciprocal arctangents of even indexed Lucas numbers, the arctangent of the ratio of any two consecutive Fibonacci numbers, with the even indexed Fibonacci number as the numerator. Similarly, by comparing identities~\eqref{equ.ric8gsn} and \eqref{equ.tjjg38v}, we obtain identity~\eqref{equ.tbk36la}, expressing, as a sum of the arctangents of reciprocal even indexed Lucas numbers, the arctangent of the ratio of any two consecutive Fibonacci numbers, with the odd indexed Fibonacci number as the numerator.~\eqref{equ.r1jgvwv} follows from~\eqref{equ.svdrzxs} and~\eqref{equ.g6b2dzg} while~\eqref{equ.e1vvx0a} is obtained by comparing~\eqref{equ.svdrzxs} and~\eqref{equ.ric8gsn}.
\end{proof}

Taking limit $n\to\infty$ in~\eqref{equ.k1d613q}, we obtain

\begin{thm}

\begin{equation}\label{equ.golzqcc}
\tan^{-1}\frac{1}{\phi}=\sum_{k=1}^{\infty}\tan^{-1}\left( \frac{1}{L_{2k}}\right)\,.
\end{equation} 
\end{thm}

\begin{rem}

It is instructive to compare~\eqref{equ.k1d613q} with the well-known result

\begin{equation}
\tan^{-1}\frac{1}{F_{2n}}=\sum_{k=n}^{\infty}\tan^{-1}\left( \frac{1}{F_{2k+1}}\right)\label{equ.cq625h1}\,,
\end{equation} 

and to compare \eqref{equ.golzqcc} with the case $n=1$ in~\eqref{equ.cq625h1}, namely,

\begin{equation}
\tan^{-1}1=\sum_{k=1}^{\infty}\tan^{-1}\left( \frac{1}{F_{2k+1}}\right)\label{equ.q0o0cvy}\,.
\end{equation}

The identity~\eqref{equ.golzqcc} was also proved in~\cite{hoggatt5} (Theorem 6).

\end{rem}

\section{BBP-type formulas}

The convergent series

\begin{equation}\label{equ.bzbrxzc}
C=\sum\limits_{k \ge 0}  {\frac{1}{{b^k }}\sum\limits_{j = 1}^l {\frac{{a_j }}{{(kl + j)^s }}} }\equiv P(s,b,l,A)\,,
\end{equation}

where $s$ and $l$ are integers, $b$ is a real number and \mbox{$A = (a_1, a_2,\ldots, a_l)$} is a vector of real numbers, defines a base-$b$ expansion of the polylogarithm constant $C$. If $b$ is an integer and $A$ is a vector of integers, then~\eqref{equ.bzbrxzc} is called a BBP type formula for the mathematical constant $C$. A BBP-type formula has the remarkable property that it allows the \mbox{$i$-th} digit of a mathematical constant to be computed without having to compute any of the previous $i-1$ digits and without requiring ultra high-precision~\cite{lord99, bailey09}. BBP-type formulas were first introduced in a 1996 paper~\cite{bbp97}, where a formula of this type for $\pi$ was given. The BBP-type formulas derived in this section will be given in the standard notation, defined by~\eqref{equ.bzbrxzc}.

\subsection{Binary BBP-type formulas for the arctangents of odd powers of the golden ratio\label{sec.binary}}

Identities~\eqref{equ.svdrzxs}, \eqref{equ.fsdxekk}, \eqref{equ.w7urgvy}, \eqref{equ.ric8gsn} and \eqref{equ.g6b2dzg} give binary BBP-type formulas for the odd powers of $\phi$, whenever binary BBP-type formulas exist for the rational numbers whose arctangents are involved. The first few BBP-type series ready identities are the following:
\begin{equation}\label{equ.k047rle}
\tan^{-1}\phi=\tan^{-1} 1+\frac{1}{2}\tan^{-1} \frac{1}{2}\quad \mbox{(identity~\eqref{equ.nhfkxe6})}\,,
\end{equation}

\begin{equation}\label{equ.e5k52tv}
\tan^{-1}\phi^3=2\tan^{-1} 1-\frac{1}{2}\tan^{-1} \frac{1}{2}\quad (\mbox{$n=1$ in~\eqref{equ.ric8gsn}})\,,
\end{equation}

\begin{equation}\label{equ.ix3ncas}
\tan^{-1}\phi^5=\tan^{-1} 1+\frac{3}{2}\tan^{-1} \frac{1}{2}\quad (\mbox{$n=2$ in~\eqref{equ.g6b2dzg}})\,,
\end{equation}

\begin{equation}\label{equ.dtctv2m}
\tan^{-1}\phi^7=3\tan^{-1} 1-\frac{3}{2}\tan^{-1} \frac{1}{2}-\tan^{-1} \frac{1}{8}\quad (\mbox{$n=2$ in~\eqref{equ.ric8gsn}})
\end{equation}

and

\begin{equation}\label{equ.o55b3gk}
\tan^{-1}\phi^9=2\tan^{-1} 1+\frac{1}{2}\tan^{-1} \frac{1}{2}-\tan^{-1} \frac{1}{4}\quad (\mbox{$n=3$ in~\eqref{equ.g6b2dzg}})\,.
\end{equation}

In obtaining the final form of~\eqref{equ.o55b3gk} we used

\[
\tan^{-1} \frac{3}{5}=\tan^{-1}1-\tan^{-1} \frac{1}{4}\,.
\]

To derive the BBP-type formulas that correspond to~\eqref{equ.k047rle}~---~\eqref{equ.o55b3gk}, we will employ the following BBP-type formulas in general bases, derived in reference~\cite{adegokejmr}:

\begin{equation}\label{equ.b0395u2}
\tan^{-1}\frac{1}{u}=\frac{1}{u^{3}}P(1,u^{4},4,(u^{2},0,-1,0))\,,
\end{equation}

\begin{equation}\label{equ.q7605l0}
\tan^{-1}\left (\frac{1}{2u-1}\right )=\frac{1}{16u^{7}}P(1,16u^{8},8,(8u^{6},8u^{5},4u^{4},0,-2u^2,-2u,-1,0))
\end{equation}

and

\begin{equation}\label{equ.dy195dg}
\tan^{-1}\left (\frac{1}{2u+1}\right )=\frac{1}{16u^{7}}P(1,16u^{8},8,(8u^{6},-8u^{5},4u^{4},0,-2u^2,2u,-1,0))\,.
\end{equation}

Using $u=2$ in~\eqref{equ.b0395u2} and $u=1$ in~\eqref{equ.q7605l0}, and forming the indicated linear combinations, identities~\eqref{equ.k047rle}~---~\eqref{equ.ix3ncas} give rise to the following BBP type formulas:

\begin{equation}
\tan^{-1}\phi=\frac{1}{16}P(1,16,8,(8,16,4,0,-2,-4,-1,0))\,,
\end{equation}

\begin{equation}
\tan^{-1}\phi^3=\frac{1}{8}P(1,16,8,(8,4,4,0,-2,-1,-1,0))
\end{equation}

and

\begin{equation}
\tan^{-1}\phi^5=\frac{1}{16}P(1,16,8,(8,32,4,0,-2,-8,-1,0))\,.
\end{equation}

Using $u=1$ in~\eqref{equ.q7605l0} and $u=2$ in~\eqref{equ.b0395u2} and expanding both series to base~$2^{12}$, length $24$, and using $u=8$ in~\eqref{equ.b0395u2} and finally forming the indicated linear combination gives the BBP type formula for $\tan^{-1}\phi^7$ as

\begin{equation}
\begin{split}
\tan^{-1}\phi^7 &=\frac{3}{4096}P(1,2^{12},24,(2048,0,1024,0,-512,-1024,-256,\\
&\qquad 0,128,0,64,0,-32,0,-16,0,8,16,4,0,-2,0,-1,0))\,.
\end{split}
\end{equation}

Using $u=1$ in~\eqref{equ.q7605l0}, $u=2$ in~\eqref{equ.b0395u2} and $u=4$ in~\eqref{equ.b0395u2}, expanding the three series to base~$256$, length $16$, and forming the indicated linear combination in~\eqref{equ.o55b3gk} gives the BBP type formula for $\tan^{-1}\phi^9$ as

\begin{equation}
\begin{split}
\tan^{-1}\phi^9 &=\frac{1}{128}P(1, 256, 16, (128, 192, 64, -128, -32, -48,\\
&\qquad -16, 0, 8, 12, 4, 8, -2, -3, -1, 0))\,.
\end{split}
\end{equation}

\subsection{ Base $5$ BBP-type formulas \label{sec.phinary}}

From the identities \eqref{equ.qmp0021}, \eqref{equ.ydokq3d} and \eqref{equ.eggvqz7} we can form the following BBP-type series ready combinations

\begin{equation}\label{equ.b9ws61o}
\begin{split}
\tan^{-1}\left(\frac{1}{\phi^2}\right)+\tan^{-1}\left(\frac{1}{\phi^4}\right)&=\tan^{-1}\left(\frac{\sqrt 5}{4}\right)\\
&=\tan^{-1}\left(\frac{1}{\sqrt 5}\right)+\tan^{-1}\left(\frac{1}{\sqrt 5^3}\right)\,,
\end{split}
\end{equation}

\begin{equation}
\tan^{-1}\left(\frac{1}{\phi^2}\right)+\tan^{-1}\left(\frac{1}{\phi^6}\right)=\tan^{-1}\left(\frac{1}{\sqrt 5}\right)
\end{equation} 

and

\begin{equation}\label{equ.rwaei1b}
\tan^{-1}\left(\frac{1}{\phi^4}\right)-\tan^{-1}\left(\frac{1}{\phi^6}\right)=\tan^{-1}\left(\frac{1}{\sqrt 5^3}\right)\,.
\end{equation} 

According to~\eqref{equ.b0395u2},

\begin{equation}
\begin{split}
&\sqrt 5\tan^{-1}\left(\frac{1}{\sqrt 5}\right)=\frac{1}{5}P(1, 25, 4, (5, 0, -1, 0))
\end{split}
\end{equation}

and

\begin{equation}
\begin{split}
&\sqrt 5\tan^{-1}\left(\frac{1}{\sqrt 5^3}\right)=\frac{1}{5^4}P(1, 5^6, 4, (5^3, 0, -1, 0))\,.
\end{split}
\end{equation}

Identities.~\eqref{equ.b9ws61o}~---~\eqref{equ.rwaei1b} therefore give rise to the following base~$5$ \mbox{BBP-type} formulas:
\begin{equation}
\begin{split}
&\sqrt 5\left\{\tan^{-1}\left(\frac{1}{\phi^2}\right)+\tan^{-1}\left(\frac{1}{\phi^4}\right)\right\}\\
&=\frac{1}{5^5}P(1, 5^6, 12, (5^5, 0, 2\cdot 5^4, 0, 5^3, 0, -5^2, 0, -10, 0, -1, 0))\,,
\end{split}
\end{equation} 

\begin{equation}
\begin{split}
&\sqrt 5\left\{\tan^{-1}\left(\frac{1}{\phi^2}\right)+\tan^{-1}\left(\frac{1}{\phi^6}\right)\right\}=\frac{1}{5}P(1, 25, 4, (5, 0, -1, 0))
\end{split}
\end{equation} 

and

\begin{equation}
\begin{split}
&\sqrt 5\left\{\tan^{-1}\left(\frac{1}{\phi^4}\right)-\tan^{-1}\left(\frac{1}{\phi^6}\right)\right\}=\frac{1}{5^4}P(1, 5^6, 4, (5^3, 0, -1, 0))\,.
\end{split}
\end{equation}

\subsection{ BBP-type formulas in base $\phi$\label{sec.phinary}}

Many BBP-type formulas in general bases were derived in reference~\cite{adegokejmr}. Base~$\phi$ formulas are easily obtained by choosing the base in any general formula of interest to be a power of $\phi$, and using the algebraic properties of $\phi$. We note that since $\phi$ is not an integer, these series are, technically speaking, not BBP-type, in the sense that they do not really lead to any digit extraction formulas, but rather correspond to base~$\phi$ expansions of the mathematical constants concerned. We now present some interesting degree~$1$ base~$\phi$ formulas. \mbox{$\phi-$nary} BBP-type formulas for $\pi$ were also derived in references~\cite{bailey01,chan,zhang}. The formulas presented here are considerably simpler and more elegant than those found in the earlier papers.

\bigskip

By setting \mbox{$n=\phi$} in identity~(27) of \cite{adegokejmr} we obtain a \mbox{$\phi-$nary} BBP-type formula for $\pi$: 

\begin{equation}\label{equ.i0d86zx}
\pi =\frac{4}{\phi^5} P(1, - \,\phi^6 ,6,(\,\phi^4 ,0,2\,\phi^2,0,1,0))\,.
\end{equation}

The base $\phi^{12}$, length $12$ version of~\eqref{equ.i0d86zx} is

\begin{equation}\label{equ.k3pgtka}
\pi = \frac{4}{\phi^{11}}P(1,\phi^{12},12,(\phi^{10}, 0, 2\phi^8, 0, \phi^6, 0, -\phi^4, 0, -2\phi^2, 0, -1, 0))\,.
\end{equation}

The following \mbox{$\phi-$nary} formulas are also readily obtained:
\[
\log \phi=  \frac{1}{\,\phi^2}P(1,\,\phi^2 ,2,(\,\phi,-1))\,,\quad(\mbox{$n=\phi$ in (28) of \cite{adegokejmr}})
\]

\[
\log 2  = \frac{1}{{\,\phi^3 }}P(1,\,\phi^3 ,3,(\,\phi^2 ,\,\phi, - 2))\,,\quad(\mbox{$n=\phi$ in (33) of \cite{adegokejmr}})
\]

\[
\tan^{-1} \left(\frac{1}{{\phi }}\right) = \frac{1}{\phi^3}P(1,\phi^4 ,4,(\phi^2,0, - 1,0))\,,\quad(\mbox{$u=\phi$ in (8) of \cite{adegokejmr}})
\]

\[
\sqrt 3\, \tan^{-1}\left (\sqrt{\frac{3}{5}}\;\right) = \frac{{ 3 }}{2\phi^5}P(1,\phi^{6} ,6,(\phi^{4} ,\phi^{3} ,0, - \phi , - 1 ,0))\,,\quad(\mbox{$n=\phi$ in (12) of \cite{adegokejmr}})
\]
\[
\sqrt 3 \tan^{-1}\left( {\frac{{\sqrt 3 }}{{\phi^3}}} \right) = \frac{{ 3 }}{2\phi^2}P(1, \phi^3 ,3,(\phi,-1,0))\,,\quad(\mbox{$n=\phi$ in (13) of \cite{adegokejmr}})
\]
\[
\begin{split}
\tan^{-1} \left( {\frac{1}{\sqrt 5}} \right) = \frac{1}{{16\phi^7}}P(1,16\,\phi^8 ,&8,(8\,\phi^6 ,8\,\phi^5 ,4\,\phi^4 ,0, - 2\,\phi^2 , - 2\,\phi, - 1,0))\,,\\
&(\mbox{$n=\phi$ in (17) of \cite{adegokejmr}})
\end{split}
\]
\[
\begin{split}
\tan^{-1} \left( {\frac{1}{\phi^3}} \right) =\frac{1}{{16\phi^7}}P(1,16\,\phi^8 ,8,(8\,\phi^6 , - 8\,\phi^5 ,&4\,\phi^4 ,0, - 2\,\phi^2 ,2\,\phi, - 1,0))\,,\\
&(\mbox{$n=\phi$ in (18) of \cite{adegokejmr}})
\end{split}
\]

\[
\sqrt 2 \tan^{-1} \sqrt 2 = \frac{2}{\phi^7} P(1,\,\phi^8 ,8,(\,\phi^6 ,0,\,\phi^4 ,0, - \,\phi^2,0, - 1,0))\,,(\mbox{n=$\phi^2$ in (21) of \cite{adegokejmr}})
\]

\[
\begin{split}
27\sqrt 3\,\tan^{-1} \left( \frac{1}{\sqrt {15}} \right) = \frac{{ 3 }}{2\phi^5 }P(1, - 27\,&\phi^6 ,6,(9\,\phi^4 ,9\,\phi^3 ,6\,\phi^2 ,3\,\phi,1,0))\\
&(\mbox{$n=\phi$ in (25) of \cite{adegokejmr}})
\end{split}
\]

and

\[
\begin{split}
27\sqrt 3\, \tan^{-1} \left( \frac{1}{\phi^3\sqrt 3} \right) =\frac{{3 }}{2\phi^5}P(1, - 27\,&\phi^6 ,6,(9\,\phi^4 ,-9\,\phi^3 ,6\,\phi^2 ,-3\,\phi,1,0))\\
&(\mbox{$n=\phi$ in (26) of \cite{adegokejmr}})\,.
\end{split}
\]

\section{Conclusion}

We have derived and presented interesting arctangent identities connecting the golden ratio, Fibonacci numbers and the Lucas numbers. Binary BBP-type formulas for the arctangents of the odd powers of the golden ratio and base~$5$ formulas for combinations of the arctangents of the reciprocal even powers of the golden ratio were derived. We also presented results for the \mbox{$\phi-$nary} expansion of some mathematical constants.

\section{Acknowledgement}

The author thanks the anonymous reviewer for a detailed review, and especially for his observation on the base~$\phi$ formulas.

\end{document}